\documentclass[12pt]{article}
\usepackage{ulem}
\usepackage{graphicx,amsmath,amsfonts,amssymb,color,mathrsfs, amsmath,amsthm}
\usepackage[colorlinks, citecolor=red]{hyperref}
\usepackage{makecell}

\usepackage{xcolor}
\usepackage{hyperref}
\usepackage{amsmath}
\allowdisplaybreaks[4]
\usepackage[authoryear]{natbib}

\usepackage{soul}

\UseRawInputEncoding
\usepackage{epsfig}
\bibdata
\bibstyle

\newcommand{\chuhao}{\fontsize{19pt}{\baselineskip}\selectfont}

\numberwithin{equation}{section}

 \newtheorem{theorem}{Theorem}[section]
 \newtheorem{lemma}{Lemma}[section]
 \newtheorem{assumption}{Assumption}[section]
 \newtheorem{proposition}{Proposition}[section]
 
 \newtheorem{remark}{Remark}[section]

\title{\bf\color{black} \chuhao{Computable Bounds on the Solution to Poisson's Equation for General Harris Chains}}

\author{
Peter W. Glynn \thanks{Department of Management Science and Engineering, Institute for Computational and Mathematical Engineering, Stanford University, USA.}
\and  Na Lin \thanks{School of Mathematics and Statistics, HNP-LAMA, Central South University, China.}
\and Yuanyuan Liu \thanks{School of Mathematics and Statistics, HNP-LAMA, Central South University, China.}}

\date{}

\begin{document}
\maketitle

\begin{abstract}
Poisson's equation is fundamental to the study of Markov chains, and arises in connection with martingale representations and central limit theorems for additive functionals, perturbation theory for stationary distributions, and average reward Markov decision process problems. In this paper, we develop a new probabilistic representation for the solution of Poisson's equation, and use Lyapunov functions to bound this solution representation explicitly. In contrast to most prior work on this problem, our bounds are computable. Our contribution is closely connected to recent work of \cite{ledoux23-pe}, in which they focus their study on a special class of Harris chains satisfying a particular small set condition. However, our theory covers general Harris chains, and often provides a tighter bound. In addition to the new bound and representation, we also develop a computable uniform bound on marginal expectations for Harris chains, and a computable bound on the potential kernel representation of the solution to Poisson's equation.

\end{abstract}

\section{Introduction}

Let $X=\left(X_{n}: n\geq 0\right)$ be an $S$-valued Markov chain with one-step transition kernel $P=\left(P(x, y\right):$ $ x, y \in S)$. For a generic function $h: S \rightarrow \mathbb{R}$ and measure $\mu$, let $P h, \mu P$, and $\mu h$ be the function, measure, and scalar defined by
$$ (P h)(x)=\int_S h(y) P(x, d y), $$
$$ (\mu P)(y)=\int_S \mu (d x) P(x, y),$$
and
$$ \mu h=\int_S h(y) \mu(d y). $$

Suppose that $P$ has a unique stationary distribution $\pi$, so that $\pi=\pi P$. For a measurable function $f: S \rightarrow \mathbb{R}$, let $f_c(\cdot)=f(\cdot)-\pi f $ be the ``centered'' version of $f$. We say that $g$ is a solution of \textit{Poisson's equation}
for the \textit{charge} (or \textit{forcing function}) $f_c$ if $\int_S |g(y)|P(x,dy)<\infty$ for $x\in S$ and
\begin{equation}\label{1.1}
(P - I)g = -f_c,
\end{equation}
where $I$ is the identity operator. Suppose that $e(x) \equiv 1$ for $x \in S$. We note that if $g$ solves \eqref{1.1}, then so does $g + c e$ for any $c \in \mathbb{R}$, so that (at best) solutions to Poisson's equation can be unique only up to an additive constant.

In this paper, we show that in the presence of suitable Lyapunov functions $v_1, v_2$ and a small set $C$, we can obtain a simple and computable bound on the (canonical) solution $g^{*}$ to \eqref{1.1}. This result complements \cite{glynn96}, in which it is shown that $g^*$ can be bounded in terms of $v_1$, $v_2$, and the small set $C$. However, the method applied there does not lend itself to computing explicit bounds on $g^*$ because it starts by analyzing the geometrically sampled version of $X$, namely $(X_{\xi_i} : i \geq 0)$,
where $\xi_0 = 0$ and $\{\xi_i - \xi_{i-1}\}_{i \geq 1}$ are independent and identically distributed (i.i.d.) geometric random variables (rv's). In contrast, we construct a probabilistic representation for $g^*$ directly in terms of $X$ (see Theorem \ref{theorem1}), thereby allowing for simpler and tighter bounds that are, importantly, readily computable.

Poisson's equation is a fundamental tool in the analysis of Markov chains. It arises naturally within the limit theory for Markov chains (see the survey by \cite{jones2004} and recent work of \cite{hofstadler2024almost}), the construction of martingales associated with Markov-dependent additive functionals (\cite{maigret1978}), gradients of steady-state performance measures (\cite{rhee2023lyapunov}), and in connection with the optimality equation for average reward/average cost stochastic control problems; see \cite{ross2014}.

\cite{glynn2002hoeffding} derive a computable bound on Poisson's equation for uniformly ergodic Markov chains, and
apply it to derive a Hoeffding inequality for such Markov chains. More recently, \cite{ledoux23-pe} also derive a closely related computable bound for the solution of Poisson's equations for Harris chains, but focus their analysis on the case in which a strong minorization condition on the transition kernel is satisfied.
In the notation of our Assumption \ref{assume1}, they primarily work under the assumption that $m=1$.
In their Remark 2.1, Herv\'e and Ledoux show how to extend their analysis for $m=1$ to the setting in which $m>1$. However, because their extension proceeds via an analysis of $(P^m-I)g=-f_c$ rather than through \eqref{1.1} directly (as in our analysis), their results for the general case are not as sharp as ours; see Remark \ref{remark-Herve}.
A major contribution of our theory is that it is intended to cover general Harris chains.

Our development of a direct bound on \eqref{1.1} requires a significantly different analysis.
Among the examples in which $m \geq 2$ appears is the Kiefer-Wolfowitz waiting time process for the G/G/s multi-server queue (see \cite{kiefer1955theory}) and typical generalized semi-Markov processes, a class of models used to represent discrete-event simulations; see \cite{henderson2001regenerative}.
Our general argument requires new ideas relative to the $m = 1$ setting because the shift invariance in the random time $\tau$ used to define the solution to Poisson's equation when $m=1$ fails when $m \geq 2$. This failure of shift invariance motivates our development of a new technique based on renewal equation ideas for proving that our probabilistic representation $g^*$ is indeed a solution of Poisson's equation (see the proof of Theorem \ref{theorem1}).

Along the way to proving Theorems \ref{theorem1} and \ref{theorem2}, we also develop a computable and uniform bound on
$\mathbb{E} \big[ f(X_n) | X_0 = x \big] $
(uniform in $n$) that is of independent interest.
All the above theory is developed in Section \ref{sec2}. In Section \ref{sec3}, we apply our theory by developing a computable bound on the potential
\[
\sum_{n=0}^{\infty} \mathbb{E} \big[ f_c(X_n) | X_0 = x \big].
\]
This potential appears in many settings, especially in bounding the bias of a steady-state simulation initialized at $X_0 = x$; see \cite{asmussen2007stochastic}.
Finally, in Section \ref{sec-example}, we present a detailed example to compare our results with those of \cite{ledoux23-pe}.

\section{A Computable Bound on the Solution of Poisson's Equation}\label{sec2}

We assume throughout the remainder of this paper that $f: S \to \mathbb{R}_+$ is non-negative. (If $f$ is of mixed sign, we can apply our bounds separately to the positive and negative parts of $f$.) For $x\in S$, let $\mathbb{P}_x(\cdot)=\mathbb{P}(\cdot|X_0=x)$ and let $\mathbb{E}_x[\cdot]$ be the expectation associated with $\mathbb{P}_x(\cdot)$. Our key assumption is:
\begin{assumption}\label{assume1}
There exists a non-empty subset $C \subset S$, non-negative functions $v_1, v_2: S \to \mathbb{R}_+$, a probability measure $\varphi$ on $S$, an integer $m \geq 1$, and positive constants $\lambda, b_1$, and $b_2$ for which:
\begin{enumerate}
    \item[(i)] $(P v_1)(x) \leq v_1(x) - f(x) + b_1 \mathbb{I}_C(x)$ for $x \in S$;
    \item[(ii)] $(P v_2)(x) \leq v_2(x) - 1 + b_2 \mathbb{I}_C(x)$ for $x \in S$;
    \item[(iii)] $\mathbb{P}_x(X_m\in \cdot) \geq \lambda \varphi(\cdot)$ for $x \in C$,
\end{enumerate}
where $\mathbb{I}_C(x) = 1$ or 0 depending on whether or not $x \in C$.
\end{assumption}

\begin{remark}
Assumption \ref{assume1} asserts that $C$ is a small set for $X$ and that $X$ is a positive recurrent Harris chain with $\pi f<\infty$; see \cite{meyn09}. Conversely, if $X$ is a positive recurrent Harris chain with $\pi f < \infty$, then Assumption \ref{assume1} ``almost'' holds, in the sense that there exist $C, v_1, v_2, m, \lambda, b_1, b_2$ for which Assumption \ref{assume1} is valid with (i) and (ii) holding for $\pi$-a.e. $x$.
\end{remark}

To obtain our bound, we first derive a new probabilistic representation for the solution to Poisson's equation. This probabilistic representation involves a randomized stopping time $\tau$. To construct $\tau$, let
$
T_1 = \inf \{ n \geq 0 : X_n \in C \}
$
be the first hitting time of $C$, and for $i \geq 1$, let
$
T_{i+1} = \inf \{ n \geq T_i + m : X_n \in C \},
$
so that the $T_i$'s are successive hitting times of $C$ spaced so that at least $m$ time steps elapse between each hitting time.

A key observation going back to \cite{athreya1978new} and \cite{nummelin1978splitting} is that Assumption \ref{assume1}(iii) allows one to write \(\mathbb{P}_x(X_{m}\in \cdot)\) over \(C\) as a mixture distribution, namely
\begin{equation}\label{mixture-distribution}
\mathbb{P}_x(X_{m}\in \cdot) = \lambda \varphi(\cdot) + (1 - \lambda) Q(x, \cdot),
\end{equation}
where \(Q(x,\cdot)\) is defined via \eqref{mixture-distribution}. Now suppose that $X$ has evolved up to time $T_1$. Given the mixture \eqref{mixture-distribution}, it is natural to generate a Bernoulli rv \(B_{T_1}\) with parameter \(\lambda\), and then distribute \(X_{T_1+m}\) according to \(\varphi\) if \(B_{T_1} = 1\) and according to \(Q(X_{T_1}, \cdot)\) if \(B_{T_1} = 0\). We now simulate the intermediate values \((X_{T_1+1}, \dots, X_{T_{1}+m-1})\) from the conditional distribution \(R(X_{T_1}, X_{T_1+m}, \cdot)\), where
\[
R(x, y, \cdot) = \mathbb{P}_x \big( (X_1, \dots, X_{m-1}) \in \cdot \, \big|\, X_m = y \big)
\]
for $x, y \in S$. Similarly, having simulated $X$ to time $T_n$ (using $B_{T_1}, \dots, B_{T_{n-1}}$), we generate the Bernoulli($\lambda$) rv $B_{T_n}$, generate $X_{T_n + m}$ using either $\varphi$ or $Q(X_{T_n},\cdot)$ depending on whether or not $B_{T_n}$ equals 1 or $0$, generate $(X_{T_n + 1}, \dots$, $X_{T_{n} +m - 1})$ using $R(X_{T_n}, X_{T_n + m},\cdot)$, and then simulate the path of $X$ from $T_{n}+m$ to $T_{n+1}$ using the one-step transition kernel $P$. Note that this construction preserves the distribution of $(X_n : n \geq 0)$.

We now let
$
\beta = \inf \{ n \geq 1 : B_{T_n} = 1 \}
$
and set
\begin{equation}\label{tau-definition}
\tau = T_\beta + m.
\end{equation}
Observe that
\begin{equation}\label{2.2}
(X_0,\dots,X_{T_{\beta}}, \tau) \quad  \mbox{and}\quad (X_\tau, X_{\tau + 1}, \dots)
\end{equation}
are independent and
\begin{equation}\label{2.3}
\mathbb{P}_x \big( (X_{\tau}, X_{\tau+1}, \dots) \in \cdot \big) = \mathbb{P}_\varphi \big( (X_0, X_1, \dots) \in \cdot \big),
\end{equation}
where
$
\mathbb{P}_\mu(\cdot) \triangleq \int_S \mu(dx) \mathbb{P}_x(\cdot)
$
for an arbitrary probability measure $\mu$ on $S$. (Similarly,
$
\mathbb{E}_\mu[\cdot] = \int_S \mu(dx) \mathbb{E}_x[\cdot].
$)
The randomized stopping time $\tau$ makes $X$ a \textit{wide-sense regenerative process} (see \cite{thorisson2000coupling}).

The key tool in our method is the Comparison Theorem, typically stated for any stopping time (see p.343 of \cite{meyn09}). In the following, we establish its extension to the case of the randomized stopping time $\tau$ as defined in \eqref{tau-definition}.

\begin{lemma}[Generalized Comparison Theorem]\label{generalized-comparison}
Suppose that there exist nonnegative functions $v$, $f$ and $s$ on $S$ such that for any $x\in S$,
\[
(Pv)(x)\leq v(x)-f(x)+s(x).
\]
Then for the randomized stopping time $\tau$ defined in \eqref{tau-definition}, we have
\[
\mathbb{E}_x \sum_{i=0}^{\tau-1} f(X_i) \leq v(x)+\mathbb{E}_x \sum_{i=0}^{\tau -1}s(X_i).
\]
\end{lemma}

\begin{proof}
To simplify the argument, assume we generate the Bernoulli rv's at every time step rather than just at the $T_j$'s but only use the $B_{T_j}$'s to implement the mixture step. These additional Bernoulli rv's have no effect on the joint distribution of $X$ and $\tau$.
We define the filtration
\[
\mathcal{F}_n = \sigma((X_j, B_j) : 0 \leq j \leq n).
\]
On \(\{\tau = k\}\), we see that for \(n \geq k\),
\[
\mathbb{E}_x[v(X_n) \mid \mathcal{F}_k] = (P^{n-k} v)(X_k),
\]
which implies that for \(n \geq k\),
\[
\mathbb{E}_x[v(X_{n+1}) \mathbb{I}(\tau \leq k) \mid \mathcal{F}_{n}] = (P v)(X_{n}) \mathbb{I}(\tau \leq k).
\]
Our assumption on $v$ ensures that $v(X_n)$ is $\mathbb{P}_x$-integrable for $n \geq 0$, so $v(X_i)-Pv(X_{i-1})$ is a martingale difference for $i\geq 1$. Hence,
\begin{align*}
&\quad \mathbb{E}_x  \sum_{i=1}^{\tau \wedge n}  \left[ v(X_i) - (Pv)(X_{i-1}) \right]\\
&= \sum_{i=1}^{n} \mathbb{E}_x \left[ v(X_i) - (Pv)(X_{i-1}) \right] \mathbb{I}(\tau \geq i)\\
&= \sum_{i=1}^{n} \mathbb{E}_x \left[ v(X_i) - (Pv)(X_{i-1}) \right]- \sum_{i=1}^{n} \mathbb{E}_x \left[ v(X_i) - (Pv)(X_{i-1}) \right] \mathbb{I}(\tau<i)\\
&= 0 - \sum_{i=1}^{n} \mathbb{E}_x \mathbb{E}_x \left[ v(X_i) \mathbb{I}(\tau\leq i-1) \mid \mathcal{F}_{i-1} \right]+ \sum_{i=1}^{n} \mathbb{E}_x (Pv)(X_{i-1}) \mathbb{I}(\tau\leq i-1)\\
&= - \sum_{i=1}^{n} \mathbb{E}_x  (Pv)(X_{i-1}) \mathbb{I}(\tau\leq i-1) + \sum_{i=1}^{n} \mathbb{E}_x (Pv)(X_{i-1}) \mathbb{I}(\tau\leq i-1)= 0.
\end{align*}
Thus, we obtain
\[
\mathbb{E}_x \sum_{i=1}^{\tau \wedge n} v(X_i) = \mathbb{E}_x \sum_{i=1}^{\tau \wedge n} (Pv)(X_{i-1}) .
\]
The rest of the proof follows the argument on p.265 of \cite{meyn09}.
\end{proof}

\begin{remark}
Note that the Comparison Theorem does not generally hold for all randomized stopping times. In particular, for \(n> T_\beta\),
\[
\mathbb{E}_x[v(X_n) \mid \mathcal{F}_{T_\beta}] \neq (P^{n-T_\beta} v)(X_{T_\beta}).
\]
This discrepancy arises because the post \(T_\beta\)-chain (i.e. \((X_{T_\beta + k} : k \geq 0)\)) does not evolve according to the transition kernel \(P\). Instead, for \(T_\beta < n \leq T_{\beta}+m\),
\[
\mathbb{E}_x[v(X_n) \mid \mathcal{F}_{T_\beta}] = \int_S \int_S \varphi(dy) R_{n - T_\beta}(X_{T_\beta}, y, dz) v(z),
\]
where \(R_j(x, y, dz) = \mathbb{P}_x(X_j \in dz \mid X_{m} = y)\).
\end{remark}

\begin{proposition}\label{proposition1}
Under Assumption \ref{assume1}, $\varphi v_i<\infty$ for $i=1, 2$ and
\begin{align}
\label{2.4}   \mathbb{E}_x  \sum_{j=0}^{\tau - 1} f(X_j)  &\leq v_1(x) + \frac{b_1 m}{\lambda},\\
\label{2.5}    \mathbb{E}_x \tau &\leq v_2(x) + \frac{b_2 m}{\lambda}, \\
\label{2.6}    \mathbb{E}_\varphi  \sum_{j=0}^{\tau - 1} f(X_j)  &\leq \delta_1\triangleq \min \left\{ \inf_{y \in S} v_1(y) + \frac{2b_1 m}{\lambda}, \varphi v_1+ \frac{b_1 m}{\lambda} \right\},\\
\label{2.7}    \mathbb{E}_\varphi \tau &\leq \delta_2 \triangleq \min \left\{ \inf_{y \in S} v_2(y) + \frac{2b_2 m}{\lambda}, \varphi v_2+ \frac{b_2 m}{\lambda} \right\}.
\end{align}
\end{proposition}

\begin{proof} In view of the fact that $X$ does not visit $C$ between $T_j+m$ and $T_{j+1}$, Lemma \ref{generalized-comparison} yields the inequality
\begin{align*}
\mathbb{E}_x  \sum_{j=0}^{\tau - 1} f(X_j)  &\leq v_1(x) + b_1 \mathbb{E}_x \sum_{j=0}^{\tau - 1} \mathbb{I}_C(X_{j }) \\
&= v_1(x) + b_1 \mathbb{E}_x  \sum_{k=1}^{\beta} \sum_{j=0}^{m - 1} \mathbb{I}_C(X_{T_k+j }) \\
&\leq v_1(x)+b_1 m \mathbb{E}_x \beta = v_1(x) + \frac{b_1 m}{\lambda}.
\end{align*}
Putting $f=e$, we similarly obtain the second inequality.

For the third inequality, suppose we apply the same algorithm as used to construct $\tau$ to the path $(X_{\tau+j}: j \geq 0)$, thereby constructing a randomized stopping time $\tau_2$ such that $X_{\tau_2}$ has distribution $\varphi$ and is independent of $\tau_2$. We again apply Lemma \ref{generalized-comparison}, yielding the inequality
\[
\mathbb{E}_x \sum_{j=0}^{\tau_2 - 1} f(X_j)  \leq v_1(x) + \frac{2b_1 m}{\lambda},
\]
for $x \in S$. Also, for any $x \in S$,
\[
\mathbb{E}_x  \sum_{j=0}^{\tau_2 - 1} f(X_j)  \geq \mathbb{E}_x  \sum_{j=\tau}^{\tau_2 - 1} f(X_j)  = \mathbb{E}_{\varphi}  \sum_{j=0}^{\tau - 1} f(X_j) ,
\]
and hence
\[
\mathbb{E}_\varphi  \sum_{j=0}^{\tau - 1} f(X_j)  \leq \inf_{x \in S} \mathbb{E}_x  \sum_{j=0}^{\tau_2 - 1} f(X_j) \leq \inf_{x \in S} v_1(x)+\frac{2b_1 m}{\lambda}.
\]

Alternatively, note that Assumption \ref{assume1} implies that
\[
P v_i \leq v_i + b_i e \quad \text{for } i = 1, 2,
\]
from which it follows that
\begin{equation}\label{2.7a}
P^n v_i \leq v_i + n b_i e,
\end{equation}
for $i=1,2$ and $n\geq 0$.
Consequently, for $x \in C$,
\[
\varphi v_i \leq \frac{(P^m v_i)(x)}{\lambda} \leq \frac{1}{\lambda}(v_i(x)+mb_i) <\infty
\]
for $i = 1, 2$. So, integrating the upper bound \eqref{2.4} with respect to $\varphi$, we obtain the alternative upper bound
\[
\mathbb{E}_\varphi  \sum_{j=0}^{\tau - 1} f(X_j)  \leq \varphi v_1 + \frac{b_1 m}{\lambda}.
\]

For the final inequality \eqref{2.7}, we apply the same argument as for $f = e$.

\end{proof}

\begin{remark}
An implication of Proposition \ref{proposition1} is that $\mathbb{E}_\varphi \tau$ and $\mathbb{E}_\varphi  \sum_{j=0}^{\tau - 1} f(X_j) $ are finite.
\end{remark}

We note that Assumption \ref{assume1} also implies that $\pi f \leq b_1$; see \cite{glynn08}. Consequently, Proposition \ref{proposition1} ensures that
\begin{equation}\label{2.8}
g^*(x) = \mathbb{E}_x  \sum_{j=0}^{\tau - 1} f_c(X_j)
\end{equation}
is finite-valued for each $x \in S$. If we can establish that $g^*$ solves Poisson's equation, then Proposition \ref{proposition1} immediately yields bounds on the solution to Poisson's equation.

When $m = 1$, this is easy to argue. We first need the following result. Let
\begin{equation}\label{nu-definition}
\nu(\cdot) = \frac{\mathbb{E}_\varphi  \sum_{j=0}^{\tau - 1} \mathbb{I}(X_j \in \cdot) }{\mathbb{E}_\varphi \tau},
\end{equation}
and note that \eqref{2.7} implies that $\nu(\cdot)$ is a probability measure on $S$.

\begin{proposition}\label{proposition2}
Under Assumption \ref{assume1}, $\pi = \nu$.
\end{proposition}

\begin{proof} Observe that for $h: S \to \mathbb{R}_+$,
\begin{align*}
\mathbb{E}_\varphi  \sum_{j=0}^{\tau - 1} h(X_j) &= \varphi h + \mathbb{E}_\varphi  \sum_{j=1}^{\tau - 1} h(X_j)\\
 &= \mathbb{E}_\varphi  \sum_{j=1}^{\tau - 1} h(X_j) +\mathbb{E}_\varphi  h(X_\tau) \\
 &= \sum_{j=1}^\infty \mathbb{E}_\varphi  h(X_j) \mathbb{I}(\tau \geq  j)\\
 &= \sum_{j=1}^\infty \mathbb{E}_\varphi  h(X_j)  - \sum_{j=1}^\infty \mathbb{E}_\varphi  h(X_j) \mathbb{I}(\tau <  j)  \\
 &= \sum_{j=1}^\infty \mathbb{E}_\varphi \mathbb{E}_\varphi[ h(X_j) \big| \mathcal{F}_{j-1} ]- \sum_{j=1}^{\infty} \mathbb{E}_\varphi  \mathbb{E}_\varphi\left[ h(X_j) \mathbb{I}(\tau < j) \big| \mathcal{F}_{j-1} \right]\\
 &= \sum_{j=1}^\infty \mathbb{E}_\varphi (Ph)(X_{j-1}) - \sum_{j=1}^{\infty} \mathbb{E}_\varphi  (Ph)(X_{j-1}) \mathbb{I}(\tau < j) \\
 &= \sum_{k=0}^\infty \mathbb{E}_\varphi  (Ph)(X_k) \mathbb{I}(\tau > k) = \mathbb{E}_\varphi  \sum_{k=0}^{\tau-1} (Ph)(X_k) ,
\end{align*}
so it follows that $\nu Ph = \nu h$ for all non-negative $h$. Since $X$ has a unique stationary distribution (as a Harris chain), $\nu = \pi$.
\end{proof}

Note that we can write
\[
\sum_{j=0}^{\tau - 1} f_c(X_j) = \gamma((X_j, B_j) : j \geq 0),
\]
where $\gamma: (S \times \{0, 1\})^\infty \to \mathbb{R}$ is a deterministic (measurable) mapping. In the case of $m=1$, we have on \(\{X_0 \notin C\}\),
\begin{equation}\label{2.10}
\sum_{j=0}^{\tau-1} f_c(X_j) = f_c(X_0) + \gamma((X_{j+1}, B_{j+1}) : j \geq 0),
\end{equation}
whereas on \(\{X_0 \in C\}\),
\begin{equation}\label{2.11}
\sum_{j=0}^{\tau-1} f_c(X_j) = f_c(X_0) + \mathbb{I}(B_0 = 0)\gamma((X_{j+1}, B_{j+1}) : j \geq 0).
\end{equation}

Taking expectations in \eqref{2.10} and \eqref{2.11} and applying the Markov property, we obtain
\begin{equation}\label{2.12}
g^*(x) = f_c(x) + \mathbb{E}_x g^*(X_1)
\end{equation}
for \(x \notin C\), while
\begin{equation}\label{2.13}
g^*(x) = f_c(x) + (1 - \lambda) \int_S Q(x, dy) g^*(y),
\end{equation}
for \(x \in C\). But \(\pi f_c = 0\), so \(\nu f_c = 0\) (see Proposition \ref{proposition2}) and hence \(\varphi g^* = 0\). So, we can rewrite \eqref{2.13} as
\begin{align*}
g^*(x) & = f_c(x) + (1 - \lambda) \int_S Q(x, dy) g^*(y) + \lambda \varphi g^*\\
 & = f_c(x) + (P g^*)(x)
\end{align*}
for $x\in C$, proving that \(g^*\) is a solution of Poisson's equation for \(m = 1\).

However, for \( m \geq 2 \), the above approach fails, because the path-shifting property in \eqref{2.11} fails. Note that the ``cycle''
\( (X_0, \dots, X_{\tau-1}) \) has alternating ``phases'', alternating between phases in which \( X \) is generated by conditioning on the ``endpoint conditional distribution'' \( R \) and phases where $X$ is generated step-by-step using $P$.
When $X$ is in a phase in which $R$ is generating the path segment $(X_{T_i +1},\ldots,X_{T_i +m-1})$ and \( X \) enters \( C \), the Bernoulli coin toss mechanism is temporarily disabled and is not free to schedule another
mixture allocation between \( \varphi \) and \( Q \).
As a consequence, when \( m \geq 2 \) and \( X_0 \in C \), there is no simple shift representation as in \eqref{2.11}.

Instead, we use a new and different argument here that exploits properties \eqref{2.2} and \eqref{2.3}.
A first step is the following new computable bound on the ``marginal'' expectation \( \mathbb{E}_x f(X_n) \), which is uniform in \( n \).

\begin{proposition}\label{proposition3}
Under Assumption \ref{assume1},
\begin{equation}\label{2.14}
\mathbb{E}_x f(X_n) \leq v_1(x) + \frac{b_1 m}{\lambda} + \delta_1
\end{equation}
for \( x \in S \), where $\delta_1$ is defined as in Proposition \ref{proposition1}.
\end{proposition}

\begin{proof}
Put \( a_n = \mathbb{E}_\varphi f(X_n) \) and note that properties \eqref{2.2} and \eqref{2.3} allow us to show that \( (a_n: n\geq 0) \) satisfies the renewal equation
\begin{align*}
a_n &= \mathbb{E}_\varphi f(X_n) \mathbb{I}(\tau > n) + \sum_{j = 1}^{n} \mathbb{E}_\varphi f(X_{\tau+n-j}) \mathbb{I}(\tau = j)\\
&=\mathbb{E}_\varphi f(X_n) \mathbb{I}(\tau > n) + \sum_{j = 1}^{n} \mathbb{E}_\varphi f(X_{n-j}) \mathbb{P}_\varphi(\tau=j)\\
&=b_n + \sum_{j=1}^n a_{n-j}p_j,
\end{align*}
where \( b_n = \mathbb{E}_\varphi f(X_n) \mathbb{I}(\tau > n) \) and $p_j=\mathbb{P}_\varphi(\tau=j)$.
Consequently, we have
\[
a_n = \sum_{j=0}^n b_{n-j} u_j,
\]
where \( (u_j: j \geq 0 ) \) is the renewal sequence associated with \( (p_j: j \geq 0) \). The term \( u_n \) can be interpreted as the probability that a point from the point process \( \mathcal{P} \) falls at the integer \( n \), where \( \mathcal{P} \) is defined as a renewal point process with i.i.d. inter-point distances $\tau$ following the distribution $( p_j : j \geq 0 )$ and an initial point at 0. As such, $u_n\leq 1$ for \( n \geq 0 \). Let \( a \wedge s  \triangleq \min\{a, s\} \), then we have
\begin{align*}
a_n &\leq \sum_{j = 0}^{n} \mathbb{E}_\varphi f(X_{j}) \mathbb{I}(\tau > j) = \mathbb{E}_\varphi  \sum_{j = 0}^{(\tau-1) \wedge  n } f(X_j) \leq \mathbb{E}_\varphi \sum_{j=0}^{\tau-1} f(X_j)
\end{align*}
for \( n \geq 0 \). For $x\in S$, observe that
\begin{align*}
\mathbb{E}_x f(X_n) &= \mathbb{E}_x f(X_n) \mathbb{I}(\tau > n) + \sum_{j = 0}^{n} \mathbb{E}_x f(X_{\tau+n-j}) \mathbb{I}(\tau = j)\\
&\leq \mathbb{E}_x \sum_{j=0}^{\tau-1} f(X_j) + \sum_{j = 0}^{n} \mathbb{E}_\varphi f(X_{n-j}) \mathbb{P}_x(\tau = j)\\
&\leq v_1(x) + \frac{b_1 m}{\lambda} + \max_{k \geq 0} \mathbb{E}_\varphi f(X_k) \cdot \sum_{j = 0}^{n} \mathbb{P}_x(\tau = j)\\
&\leq v_1(x) + \frac{b_1 m}{\lambda} + \mathbb{E}_\varphi \sum_{j=0}^{\tau-1} f(X_j) \\
& \leq v_1(x) + \frac{b_1 m}{\lambda} + \delta_1,
\end{align*}
where Proposition \ref{proposition1} is used for both the third and fifth inequalities above.
\end{proof}

The following result establishes that \( g^* \) is a solution of Poisson's equation.

\begin{theorem}\label{theorem1}
Suppose that Assumption \ref{assume1} holds. Then, \( g^* \) is a solution of Poisson's equation \eqref{1.1} for the forcing function \( f_c \).
\end{theorem}
\begin{proof}
For \( n \geq 0 \), put
\[
\kappa_n(x) \triangleq \sum_{j=0}^{n} \mathbb{E}_x f_c(X_j)
\]
and note that Proposition \ref{proposition3} implies that \( \kappa_n(x) \) is finite-valued for each \( x \in S \). Then,
\begin{equation}\label{2.15}
\kappa_n(x) = f_c(x) + (P \kappa_{n-1})(x)
\end{equation}
for \( x \in S \). Furthermore,
\begin{align}
\nonumber \kappa_n(x) &= \mathbb{E}_x \sum_{j=0}^{(\tau-1)\wedge n} f_c(X_j) + \mathbb{E}_x \sum_{j=\tau}^{n} f_c(X_j) \mathbb{I}(\tau \leq n) \\
\nonumber &= \mathbb{E}_x \sum_{j=0}^{(\tau-1)\wedge n} f_c(X_j)+ \sum_{j=0}^{n} \mathbb{P}_x(\tau=j)\cdot \mathbb{E}_\varphi \sum_{k=0}^{n-j} f_c(X_k)\\
\nonumber &= \beta_n(x) + \sum_{k=0}^{n} \mathbb{E}_\varphi f_c(X_k) \mathbb{P}_x(\tau \leq n-k)\\
\nonumber &= \beta_n(x) + \sum_{k=0}^{n} \mathbb{E}_\varphi f_c(X_k) (1-\mathbb{P}_x(\tau>n-k)) \\
\label{2.16} &= \beta_n(x) + \sum_{k=0}^{n} \mathbb{E}_\varphi f_c(X_k)  - c_n(x),
\end{align}
where
\begin{align*}
\beta_n(x) &\triangleq \mathbb{E}_x \sum_{j=0}^{(\tau-1)\wedge n} f_c(X_j) ,\\
c_n(x) &\triangleq \sum_{j=0}^{n} \mathbb{P}_x(\tau > j) \mathbb{E}_\varphi f_c(X_{n-j})
\end{align*}
for $n\geq 0$.
Finally, plugging the identity \eqref{2.16} for \( \kappa_n(x) \) and \( \kappa_{n-1}(x) \) into \eqref{2.15}, we get
\begin{equation}\label{2.17}
\beta_n(x) = f_c(x) + \mathbb{E}_x \beta_{n-1}(X_1) - \mathbb{E}_\varphi f_c(X_{n}) + c_n(x) - \mathbb{E}_x c_{n-1}(X_1).
\end{equation}
A complication in using the above identity is that the equality $\mathbb{P}_x(\tau>j)=\int_S P(x,dy)\mathbb{P}_y(\tau>j-1)$ need not hold (because $\tau=m$ holds with probability $\lambda$ when $x\in C$, but $\tau=m-1$ is impossible). As a result, we apply the following argument instead.

Since \( \sum_{k=0}^{\tau-1} |f_c(X_k)| \) is \( \mathbb{P}_x \)-integrable (by Proposition \ref{proposition1}), the Dominated Convergence Theorem implies that
\begin{equation}\label{2.18}
\beta_n(x) \to g^*(x),
\end{equation}
as \( n \to \infty \), for each \( x \in S \). Furthermore, because Assumption \ref{assume1} asserts that \( (P v_i)(x) \leq v_i(x) + b_i \) for $i=1,2$, evidently
\[
\int_S P(x,dy) \left( \mathbb{E}_y \sum_{j=0}^{\tau-1} f(X_j) + \mathbb{E}_y \tau \right) < \infty
\]
for each \( x \in S \), and hence the Dominated Convergence Theorem proves that
\begin{equation}\label{2.19}
\mathbb{E}_x \beta_{n-1}(X_1) \to \mathbb{E}_x g^*(X_1)
\end{equation}
as \( n \to \infty \).

We now (temporarily) assume that \( X \) is aperiodic. Then,
\begin{equation}\label{2.20}
\mathbb{E}_\varphi f_c(X_n) \to 0
\end{equation}
as \( n \to \infty \), and for \( z \in S \),
\[
\sum_{k=0}^{\infty} \mathbb{P}_z(\tau > k) = \mathbb{E}_z \tau \leq v_2(z) + \frac{b_2 m}{\lambda}<\infty,
\]
as a result of Proposition \ref{proposition1}. In view of \eqref{2.20}, the Bounded Convergence Theorem therefore implies that \( c_n(z) \to 0 \) as \( n \to \infty \) for all \( z \in S \). Consequently, both \( c_n(x) \) and \( c_{n-1}(X_1) \) tend to 0 as \( n \to \infty \).

Assumption \ref{assume1}(ii) shows that \( v_2(X_1) \) is \( \mathbb{P}_x \)-integrable, so the Dominated Convergence Theorem therefore shows that
\[
\mathbb{E}_x c_{n-1}(X_1) \to 0
\]
as \( n \to \infty \). By sending \( n \to \infty \) in \eqref{2.17}, we therefore conclude that \( g^* \) satisfies Poisson's equation under Assumption \ref{assume1} and the assumption of aperiodicity.

For the periodic case, we use Theorem 5.4.4 of \cite{meyn09} to establish the existence of an absorbing set \( D \) that can be partitioned into \( p \) disjoint periodic subsets \( D_0, D_1, \dots, D_{p-1} \) such that \( P(x, D_{i+1}) = 1 \) for \( x \in D_i \) (\( 0 \leq i < p-1 \)), with \( P(x, D_0) = 1 \) for \( x \in D_{p-1} \).
Without loss of generality, we assume that $C\subseteq D_0$, so that $\varphi$ is fully supported on $D_r$, where $r = m \mod p$. Because $\tau$ can not occur until $m$ time units after $C$ is hit, it follows that $\mathbb{P}_{x}(\tau < \ell +m) = 0$ for $x\in D_{p-\ell}$, $0<\ell \leq p$. Also, the periodicity implies that
\[
\mathbb{P}_x(\tau> m+ \ell +j p)=\mathbb{P}_x (\tau > m+\ell +jp+i)
\]
for $0\leq i \leq p-1$, $j\geq 0$, and $x\in D_{p-\ell}$, $0<\ell \leq p$.
For \( x \in D_{p-\ell}\), let \( k_n = m+ \ell + n p - 1\), so that
\begin{align}
\nonumber c_{k_n}(x) &=
\sum_{j=0}^{m  + \ell -1} \mathbb{P}_x (\tau > j) \mathbb{E}_{\varphi} f_c (X_{k_n - j}) + \sum_{j = m  + \ell}^{k_n} \mathbb{P}_x (\tau > j) \mathbb{E}_{\varphi} f_c (X_{k_n - j})\\
\nonumber &= \sum_{j=0}^{m  + \ell -1} \mathbb{E}_{\varphi} f_c (X_{k_n - j})
+ \sum_{j=0}^{n-1} \sum_{i=0}^{p-1} \mathbb{P}_x (\tau > m + \ell + j p + i) \mathbb{E}_{\varphi} f_c (X_{(n - j) p - i-1})\\
\label{2.23new}&= \sum_{j=0}^{m  + \ell -1} \mathbb{E}_{\varphi} f_c (X_{k_n - j})
+ \sum_{j=0}^{n-1}  \mathbb{P}_x (\tau > m + \ell + j p ) \sum_{i=0}^{p-1} \mathbb{E}_{\varphi} f_c (X_{(n - j) p - i-1}).
\end{align}
Because of the periodicity,
\[
\mathbb{P}_{\varphi} (X_{n p + i} \in \cdot) \to \pi_{s(i)}(\cdot)
\]
as \( n \to \infty \), where \( s(i) \triangleq (r+i) \mod p \) and
$\pi_j (\cdot) \triangleq p \pi(\cdot \cap D_j)$, for $0 \leq j < p$, is the distribution of \( \pi \) conditioned on \( D_j \).
Furthermore,
\[
\varphi(\cdot) \leq \mathbb{E}_{\varphi}\tau \cdot \pi(\cdot),
\]
so
\begin{align*}
\mathbb{E}_{\varphi} | f_c (X_n)| \mathbb{I} (f(X_n) > w) &\leq \mathbb{E}_{\varphi} \tau \cdot \mathbb{E}_{\pi} |f_c (X_n)| \mathbb{I} (f(X_n) > w)\\
&= \mathbb{E}_{\varphi} \tau \cdot \mathbb{E}_{\pi} |f_c (X_0)| \mathbb{I} (f(X_0) > w),
\end{align*}
and hence \( ( f_c (X_n): n\geq 0 )\) is uniformly integrable under \( \mathbb{P}_{\varphi} \). So,
\begin{equation}\label{2.24new}
\mathbb{E}_{\varphi} f_c (X_{np + i}) \to \pi_{s(i)} f_c
\end{equation}
as \( n \to \infty \).
Observe that
\[
\sum_{j=0}^{\infty} \mathbb{P}_x (\tau > m +\ell + j p) \leq \mathbb{E}_x \tau < \infty,
\]
due to \eqref{2.5}, so the Bounded Convergence Theorem applied to \eqref{2.23new} yields the conclusion
\begin{align*}
c_{k_n} (x) \to \sum_{j=0}^{m + \ell -1} \pi_{s(m +\ell - j -1)} f_c
+ \sum_{j=0}^{\infty} \mathbb{P}_x (\tau > m+\ell + j p) \sum_{i=0}^{p-1} \pi_{s(p - i-1)} f_c
\end{align*}
as \( n \to \infty \).
But
\[
\sum_{i=0}^{p-1} \pi_i f_c = p \pi f_c =0,
\]
so
\begin{equation}\label{2.25new}
c_{k_n} (x) \to \sum_{j=0}^{m + \ell -1} \pi_{s(m +\ell - j -1)} f_c
\end{equation}
as \( n \to \infty \) for \( x \in D_{p-\ell} \). Similarly, for \( x \in D_{p-\ell + 1} \) (which we interpret as \( D_0 \) when \( \ell = 1 \)),
\[
c_{k_n -1} (x) \to \sum_{j=0}^{m + \ell -2} \pi_{s(m +\ell - j -2)} f_c
\]
as \( n \to \infty \). The same Dominated Convergence Theorem argument as used to justify \eqref{2.19} then proves that
\begin{equation}\label{2.26new}
\mathbb{E}_x c_{k_n -1} (X_1) \to \sum_{j=0}^{m + \ell -2} \pi_{s(m +\ell - j -2)} f_c \quad
\end{equation}
as \( n \to \infty \). But \eqref{2.25new} and \eqref{2.26new} imply that
\begin{equation}\label{2.27new}
c_{k_n} (x) - \mathbb{E}_x c_{k_n -1} (X_1) \to \pi_{s(m + \ell -1)} f_c
\end{equation}
as \( n \to \infty \). On the other hand, \eqref{2.24new} shows that
\begin{equation}\label{2.28new}
\mathbb{E}_{\varphi} f_c (X_{k_n}) \to \pi_{s(m+\ell -1)} f_c
\end{equation}
as \( n \to \infty \). Sending \( n \to \infty \) through the subsequence \( (k_n: n\geq 1) \) in \eqref{2.17} and utilizing \eqref{2.27new} and \eqref{2.28new} then proves that \( g^* \) solves Poisson's equation on \( D \).

Finally, for $x$ outside the absorbing set $D$, we note that $C$ must be contained within $D$ and \eqref{2.10} holds, so that Poisson's equation holds there also.

\end{proof}

With Proposition \ref{proposition1} and Theorem \ref{theorem1} in hand, Theorem \ref{theorem2} easily follows:

\begin{theorem}\label{theorem2}
Under Assumption \ref{assume1}, \( g^* \) is a solution of Poisson's equation \eqref{1.1}, and
\[
-b_1 \left( v_2(x) +\frac{ b_2 m}{\lambda} \right) \leq g^*(x) \leq v_1(x) + \frac{b_1 m}{\lambda},
\]
and
\[
|g^*(x)| \leq \max \left\{ v_1(x) + \frac{b_1 m}{\lambda}, b_1 \left( v_2(x) + \frac{b_2 m}{\lambda} \right) \right\}
\]
for \( x \in S \).
\end{theorem}

\begin{proof}
The non-negativity of $f$ guarantees that
\begin{equation*}
-\pi f\cdot \mathbb{E}_x \tau \leq g^*(x) \leq  \mathbb{E}_x \sum_{j=0}^{\tau-1}f(X_j).
\end{equation*}
The assertion follows immediately by Proposition \ref{proposition1} and the fact that $\pi f\leq b_1$.
\end{proof}

\begin{remark}\label{remark-Herve}
As mentioned in the Introduction, \cite{ledoux23-pe} deal with \( m > 1 \) through
consideration of the Poisson's equation for \( P^m \), namely
$(P^m - I) g = -f_c$, and then apply their results for \( m = 1 \) to this equation.
As such, they assume our Assumption \ref{assume1} with \( P^m \) replacing \( P \).
It typically is much easier to verify Lyapunov inequalities like Assumption \ref{assume1} (i) and (ii) when they involve \( P \) rather than \( P^m \) (because \( P \) is known explicitly, and \( P^m \) is typically not known in closed form), so verifying their analog to Assumption \ref{assume1} will typically be significantly more difficult.

Furthermore, the \( m \)-step chain may never hit \( C \) from \( x \) when the chain is periodic,
so \cite{ledoux23-pe} need to assume that \( X \) is an aperiodic Markov chain.
(This condition is implicit in their assumption that $P^m$ and $P$ have the same stationary distributions.)
In addition, as noted in their Remark 2.1, their bound degenerates geometrically in the parameter \( m \)
(so that, in their words, the bound is ``essentially theoretical''), because their argument is not optimized for \( m > 1 \).
\end{remark}

We note that an immediate consequence of Theorem \ref{theorem2} is that
\begin{equation}\label{2.27}
g^*(X_n) + \sum_{i=0}^{n-1} f_c(X_i), \quad n\geq 0,
\end{equation}
is a \( \mathbb{P}_x \)-martingale for each \( x \in S \). Given that \( g^* \) satisfies Poisson's equation and the \( f_c(X_i) \)'s are integrable (as a consequence of Proposition \ref{proposition3}), only the \( \mathbb{P}_x \)-integrability of \( g^*(X_n) \) needs to be verified.
Given Theorem \ref{theorem2}, this follows from the \( \mathbb{P}_x \)-integrability of \( v_i(X_n) \) for \( i = 1, 2 \). But this is a consequence of \eqref{2.7a}.

\section{A Potential Kernel Representation for the Solution of Poisson's Equation}\label{sec3}

As noted in Section \ref{sec2}, \eqref{2.27} is a martingale in the presence of Assumption \ref{assume1}. Hence,
\[
g^*(x) - \mathbb{E}_x g^*(X_n) = \sum_{i=0}^{n-1} \mathbb{E}_x f_c(X_i)
\]
for \( n \geq 0 \). If
\[
\lim_{n \to \infty} \sum_{i=0}^{n-1} \mathbb{E}_x f_c(X_i)
\]
exists and is finite-valued, then \( \mathbb{E}_x g^*(X_n) \) also converges to a finite-valued limit and
\begin{equation}\label{3.1}
g^*(x) = \lim_{n \to \infty} \sum_{i=0}^{n-1} \mathbb{E}_x f_c(X_i) + c,
\end{equation}
for some constant \( c \in \mathbb{R} \). Hence, \( g^* \) is (up to an additive constant) given by the infinite sum ``potential'' on the right-hand side of \eqref{3.1}.
When \( X \) is aperiodic, we expect that \( \mathbb{E}_x g^*(X_n) \) converges to a finite limit when \( \pi |g^*|<\infty \); the need for this extra moment condition to ensure \eqref{3.1} is discussed in greater detail in the countable state space setting in \cite{glynn2024solution}.

\begin{assumption}\label{assume2}
In addition to Assumption \ref{assume1}, we assume there exist non-negative functions \( v_3, v_4 : S \to \mathbb{R}_+ \) and constants \( b_3, b_4 \) such that:
\begin{equation}\label{drift_v3}
(P v_3)(x) \leq v_3(x) - v_1(x) + b_3 \mathbb{I}_C(x),
\end{equation}
\[
(P v_4)(x) \leq v_4(x) - v_2(x) + b_4 \mathbb{I}_C(x).
\]
\end{assumption}

\begin{theorem}\label{theorem3}
Suppose that Assumption \ref{assume2} holds. If \( X \) has period \( p \geq 1 \), then for each \( x \in S \),
\begin{equation}\label{3.2}
\tilde{g}(x) = \lim_{n \to \infty} \sum_{i=0}^{np-1} \mathbb{E}_x f_c(X_i)
\end{equation}
exists and
\begin{equation}\label{3.3}
-p b_3 - \frac{b_1 m}{\lambda} \leq \tilde{g}(x) - g^*(x) \leq b_1\left(p b_4 + \frac{b_2 m}{\lambda}\right)
\end{equation}
and
\begin{equation}\label{3.4}
|\tilde{g}(x) - g^*(x)| \leq \max \left\{ b_1\left(p b_4 + \frac{b_2 m}{\lambda}\right), p b_3 + \frac{b_1 m}{\lambda} \right\}.
\end{equation}
\end{theorem}

\begin{proof}
Assumption \ref{assume2} ensures that \( \pi v_1 \leq b_3 \) and \( \pi v_2 \leq b_4 \); see \cite{glynn08}. Fix \( x \in D_{p-\ell} \), for $1 \leq \ell \leq p$, and let $\bar{n}=\max\{i\geq 0: \ell+m+ip\leq np\}$. Then,
\begin{align}
\nonumber \mathbb{E}_x g^*(X_{np}) = \mathbb{E}_x g^*(X_{np}) &\mathbb{I}(\tau > np) \\
\label{3.5} &+ \sum_{i=0}^{\bar{n}} \mathbb{P}_x(\tau = \ell+m+ip) \mathbb{E}_\varphi  g^*(X_{(n-i)p-\ell-m})
\end{align}
for \( n\geq 1 \). Letting \( v_i \) play the role of \( f \) in Proposition \ref{proposition1}, we find that for $i=1,2$,
\begin{equation}\label{3.6}
\mathbb{E}_x  \sum_{j=0}^{\tau-1} v_i(X_j)  \leq v_{i+2}(x) + \frac{b_{i+2} m}{\lambda},
\end{equation}
so that
\[
\mathbb{E}_x v_i(X_n) \mathbb{I}(\tau > n)\leq \mathbb{E}_x  \sum_{j=0}^{\tau-1} v_i(X_j)\mathbb{I}(\tau > n) \to 0
\]
as $n \to \infty$.
Theorem \ref{theorem2} bounds \( g^* \) in terms of \( v_1 \) and \( v_2 \), so that
\begin{equation}\label{3.7}
\mathbb{E}_x g^*(X_{np}) \mathbb{I}(\tau > np) \to 0
\end{equation}
as $n \to \infty$.

Put \( a_n^* = \mathbb{E}_\varphi g^*(X_n) \) and \( b_n^* = \mathbb{E}_\varphi g^*(X_n) \mathbb{I}(\tau > n) \). Because of \eqref{3.6}, \( (b_n^*:n \geq 0) \) is absolutely summable. Also,
\[
a_n^* = b_n^* + \sum_{j=1}^{n} a_{n-j}^* p_j,
\]
so that
\[
a_n^* = \sum_{j=0}^{n} b_{n-j}^* u_{j}
\]
for \( n \geq 0 \). Applying the renewal theorem for \( (p_j: j \geq 0) \) concentrated on some multiple of \( p \), we find that \( a_{k p-\ell-m}^* \) converges to a limit as \( k \to \infty \). The Bounded Convergence Theorem and \eqref{3.7} therefore ensure that \( \mathbb{E}_x g^*(X_{np}) \) converges to a finite limit. In view of \eqref{2.27}, we find that the limit in \eqref{3.2} exists. Furthermore, for $x\in D_i$,
\[
\mathbb{E}_x g^*(X_{np}) \to \pi_{i} g^*,
\]
as \( n \to \infty \). In view of Theorem \ref{theorem2},
\begin{align}
\nonumber \pi_i g^* &\leq \pi_i v_1 +  \frac{b_1 m}{\lambda}\pi_i e\\
\label{3.8} &\leq p\pi v_1 +\frac{b_1 m}{\lambda} \leq p b_3 + \frac{b_1 m}{\lambda}
\end{align}
and
\begin{equation}\label{3.9}
\pi_i g^* \geq -b_1 \left( \pi_i v_2 + \frac{b_2 m}{\lambda} \right) \geq -b_1 \left( p b_4 + \frac{b_2 m}{\lambda} \right).
\end{equation}
Since
\[
\tilde{g}(x) - g^*(x) = - \pi_{i} g^*,
\]
Theorem \ref{theorem2}, \eqref{3.8}, and \eqref{3.9} yield the bounds \eqref{3.3} and \eqref{3.4}.
\end{proof}

\section{An Illustrative Example}\label{sec-example}

In this section, we provide an example to allow the comparison of our bound in Theorem \ref{theorem2} to that provided by Theorem 2.3 of \cite{ledoux23-pe}. In particular, we consider the delay sequence (or waiting time sequence) \((W_n : n \geq 0)\) associated with
the single-server GI/G/1 queue; see p.267 of \cite{asmussen2003applied} for details.
This Markov chain satisfies the stochastic recursion
\[
W_{n+1} = [W_n + Z_{n}]^+,
\]
where \([x]^+ = \max\{x, 0\}\), and \((Z_n : n \geq 0)\) is an i.i.d. sequence independent of \(W_0\). We assume that \(Z_0\) has a continuous positive density \(h_Z(\cdot)\) with \(\mathbb{E} Z_0 < 0\) and \(\mathbb{E} Z_0^2 < \infty\). Under these conditions, \((W_n : n \geq 0)\) has a unique stationary distribution \(\pi\) with \(\pi f < \infty\), where \(f(x) = x\) for \(x \geq 0\); see Theorem 2.1 on p.270 of \cite{asmussen2003applied}.

The transition kernel $P$ of this Markov chains can be expressed as
\[
P(x,dy)= \mathbb{I}_{0}(dy)\int_{-\infty}^{-x}h_Z(w)dw + \mathbb{I}_{(0,\infty]}(y)h_Z(y-x)dy,
\]
for $x,y \geq 0$.
We put $v_1(x) = c_1 x^2 \vee 1$, $v_2(x) = v_1(x)$, where $c_1>0$ and $a\vee b \triangleq \max\{a,b\}$.
Note that for $i=1,2$,
\[
(Pv_i)(x) - v_i(x) \sim 2c_1 x \mathbb{E} Z_0
\]
as \(x \to \infty\), where \(\sim\) means that the ratio of the left-hand side to the right-hand side converges to 1. If we put \(c_1 = \frac{\kappa}{2|\mathbb{E} Z_0|} \) with $\kappa>1$, then there exists an interval \(C = [0, x_0]\) for which Assumption \ref{assume1} holds, where \(\varphi(dy) = \frac{\phi(dy)}{\lambda} \), \(m = 1\),
\begin{align*}
\phi(dy) &= \mathbb{I}_{0}(dy) \int_{-\infty}^{-x_0}h_Z(w)dw + \inf_{0 \leq x \leq x_0} h_Z(y - x)dy,\\
\lambda &= \int_{0}^{\infty} \phi(dw),\\
b_1 & = b_2 = \sup_{0 \leq x \leq x_0} ((P v_1)(x) - v_1(x) +f(x) \vee 1).
\end{align*}
Our bound on the solution $g^*$ (see Theorem \ref{theorem2}), valid under the assumption \(\mathbb{E} Z_0^2 < \infty\), is
\begin{equation}\label{4.1}
- b_1 \left( \frac{\kappa x^2}{2|\mathbb{E} Z_0 |} + \frac{b_1}{\lambda} \right) \leq g^*(x) \leq \frac{\kappa x^2}{2 |\mathbb{E} Z_0 |} + \frac{b_1}{\lambda}
\end{equation}
for \(x \) large.

On the other hand, the bound of \cite{ledoux23-pe} (see their discussion on p.11) requires the same assumption \(\mathbb{E} Z_0^2 < \infty\) and takes the form
\begin{equation}\label{4.2}
|g^*(x)| \leq a (1 + b_1)\cdot \frac{\kappa x^2}{ 2|\mathbb{E} Z_0 |}
\end{equation}
for \(x \) large, where
\[
a = 1 + \max\left\{0, \frac{b_1}{\lambda} - \varphi v_1\right\}.
\]
We observe that for this example in which $m=1$, our upper bound in \eqref{4.1} is asymptotic to \( \max\{1,b_1\}\frac{\kappa x^2}{2|\mathbb{E} Z_0 |}\) as \(x \to \infty\), whereas \eqref{4.2} is asymptotically at least as large as \((1+b_1)\frac{\kappa x^2}{2|\mathbb{E} Z_0 |}\), so our bound is tighter for large \(x\).

\normalem
\bibliographystyle{plainnat}
\bibliography{ref}

\begin{thebibliography}{18}
\providecommand{\natexlab}[1]{#1}
\providecommand{\url}[1]{\texttt{#1}}
\expandafter\ifx\csname urlstyle\endcsname\relax
  \providecommand{\doi}[1]{doi: #1}\else
  \providecommand{\doi}{doi: \begingroup \urlstyle{rm}\Url}\fi

\bibitem[Asmussen(2003)]{asmussen2003applied}
S{\o}ren Asmussen.
\newblock \emph{Applied Probability and Queues}.
\newblock Springer-Verlag, New York, 2003.

\bibitem[Asmussen and Glynn(2007)]{asmussen2007stochastic}
S{\o}ren Asmussen and Peter~W. Glynn.
\newblock \emph{Stochastic {S}imulation: {A}lgorithms and {A}nalysis}.
\newblock Springer, New York, 2007.

\bibitem[Athreya and Ney(1978)]{athreya1978new}
Krishna~B. Athreya and Peter Ney.
\newblock A new approach to the limit theory of recurrent {M}arkov chains.
\newblock \emph{Transactions of the American Mathematical Society},
  245:\penalty0 493--501, 1978.

\bibitem[Glynn and Infanger(2024)]{glynn2024solution}
Peter~W. Glynn and Alex Infanger.
\newblock Solution representations for {P}oisson's equation, martingale
  structure, and the {M}arkov chain central limit theorem.
\newblock \emph{Stochastic Systems}, 14\penalty0 (1):\penalty0 47--68, 2024.

\bibitem[Glynn and Meyn(1996)]{glynn96}
Peter~W. Glynn and Sean~P. Meyn.
\newblock A {L}iapounov bound for solutions of the {P}oisson equation.
\newblock \emph{The Annals of Probability}, 24\penalty0 (2):\penalty0 916--931,
  1996.

\bibitem[Glynn and Ormoneit(2002)]{glynn2002hoeffding}
Peter~W. Glynn and Dirk Ormoneit.
\newblock Hoeffding's inequality for uniformly ergodic {M}arkov chains.
\newblock \emph{Statistics \& Probability Letters}, 56\penalty0 (2):\penalty0
  143--146, 2002.

\bibitem[Glynn and Zeevi(2008)]{glynn08}
Peter~W. Glynn and Assaf Zeevi.
\newblock Bounding stationary expectations of {M}arkov processes.
\newblock In \emph{Markov Processes and Related Topics: a Festschrift for
  Thomas G. Kurtz}, volume~4, pages 195--215. 2008.

\bibitem[Henderson and Glynn(2001)]{henderson2001regenerative}
Shane~G. Henderson and Peter~W. Glynn.
\newblock Regenerative steady-state simulation of discrete-event systems.
\newblock \emph{ACM Transactions on Modeling and Computer Simulation (TOMACS)},
  11\penalty0 (4):\penalty0 313--345, 2001.

\bibitem[Herv{\'e} and Ledoux(2025)]{ledoux23-pe}
Lo{\"\i}c Herv{\'e} and James Ledoux.
\newblock Computable bounds for solutions to {P}oisson's equation and
  perturbation of {M}arkov kernels.
\newblock \emph{To appear in Bernoulli}, 2025.

\bibitem[Hofstadler et~al.(2024)Hofstadler, {\L}atuszy{\'n}ski, Roberts, and
  Rudolf]{hofstadler2024almost}
Julian Hofstadler, Krzysztof {\L}atuszy{\'n}ski, Gareth~O. Roberts, and Daniel
  Rudolf.
\newblock Almost sure convergence rates of adaptive increasingly rare {M}arkov
  chain {M}onte {C}arlo.
\newblock \emph{arXiv:2402.12122}, 2024.

\bibitem[Jones(2004)]{jones2004}
Galin~L. Jones.
\newblock On the {M}arkov chain central limit theorem.
\newblock \emph{Probability Surveys}, 1:\penalty0 299--320, 2004.

\bibitem[Kiefer and Wolfowitz(1955)]{kiefer1955theory}
Jack Kiefer and Jacob Wolfowitz.
\newblock On the theory of queues with many servers.
\newblock \emph{Transactions of the American Mathematical Society}, 78\penalty0
  (1):\penalty0 1--18, 1955.

\bibitem[Maigret(1978)]{maigret1978}
Nelly Maigret.
\newblock Th{\'e}or{\`e}me de limite centrale fonctionnel pour une cha{\^\i}ne
  de {M}arkov r{\'e}currente au sens de {H}arris et positive.
\newblock \emph{Annales de l'Institut Henri Poincar{\'e}. Section B. Calcul des
  probabilit{\'e}s et statistiques}, 14\penalty0 (4):\penalty0 425--440, 1978.

\bibitem[Meyn and Tweedie(2009)]{meyn09}
Sean~P. Meyn and Richard~L. Tweedie.
\newblock \emph{Markov {C}hains and {S}tochastic {S}tability}.
\newblock Cambridge University Press, Cambridge, 2009.

\bibitem[Nummelin(1978)]{nummelin1978splitting}
Esa Nummelin.
\newblock A splitting technique for {H}arris recurrent {M}arkov chains.
\newblock \emph{Zeitschrift f{\"u}r Wahrscheinlichkeitstheorie und verwandte
  Gebiete}, 43\penalty0 (4):\penalty0 309--318, 1978.

\bibitem[Rhee and Glynn(2023)]{rhee2023lyapunov}
Chang-Han Rhee and Peter~W. Glynn.
\newblock Lyapunov conditions for differentiability of {M}arkov chain
  expectations.
\newblock \emph{Mathematics of Operations Research}, 48\penalty0 (4):\penalty0
  2019--2042, 2023.

\bibitem[Ross(2014)]{ross2014}
Sheldon~M. Ross.
\newblock \emph{Introduction to {P}robability {M}odels}.
\newblock Academic Press, Amsterdam, 2014.

\bibitem[Thorisson(2000)]{thorisson2000coupling}
Hermann Thorisson.
\newblock \emph{Coupling, {S}tationarity, and {R}egeneration}.
\newblock Springer, New York, 2000.

\end{thebibliography}

\end{document}